\documentclass[10pt]{amsart}

\usepackage{amscd, verbatim, amssymb, multirow, multicol}
\usepackage{amsmath}
\usepackage{amsthm}
\usepackage{graphicx}
\usepackage{hyperref, todonotes}

\newcommand{\End}{\ensuremath {\mathcal End}}
\newcommand{\Ext}{\ensuremath {\mathrm {Ext}}}
\newcommand{\Tors}{\ensuremath {\mathrm {Tors\,}}}
\newcommand{\Hom}{\ensuremath {\mathrm {Hom}}}
\newcommand{\cHom}{\ensuremath {\mathcal Hom}}

\newcommand{\Aut}{\ensuremath {\mathrm {Aut}}}

\newcommand{\Pic}{\ensuremath {\mathrm {Pic}}}

\newcommand{\lcm}{\ensuremath {\mathrm{lcm}}}

\newcommand{\rank}{\ensuremath {\mathrm {rank\,}}}
\newcommand{\Cl}{\ensuremath {\mathrm {Cl\,}}}
\newcommand{\length}{\ensuremath {\mathrm {length\,}}}

\newcommand{\MM}{\ensuremath {\mathcal M}}
\newcommand{\Mbar}{\ensuremath {\overline{\mathcal M}}}

\newcommand{\OO}{\ensuremath {\mathcal O}}

\newcommand{\II}{\ensuremath {\mathcal I}}

\newcommand{\Q}{\ensuremath {\mathbb Q}}
\newcommand{\C}{\ensuremath {\mathbb C}}
\newcommand{\R}{\ensuremath {\mathbb R}}

\newcommand{\Pro}{\ensuremath {\mathbb P}}
\newcommand{\Z}{\ensuremath {\mathbb Z}}

\newcommand{\X}{\ensuremath {\widetilde X}}
\newcommand{\wZ}{\ensuremath {\widetilde Z}}

\newtheorem{thm}{Theorem}[section]
\newtheorem{prop}[thm]{Proposition}
\newtheorem{lem}[thm]{Lemma}
\newtheorem{cor}[thm]{Corollary}

\theoremstyle{definition}
\newtheorem{defin}[thm]{Definition}
\newtheorem{remark}[thm]{Remark}

\begin{document}
\title[Degenerations of Godeaux surfaces and vector bundles]{Degenerations of Godeaux surfaces and exceptional vector bundles}
\author{Anna Kazanova}
\address{Department of Mathematics, University of Georgia, Athens, GA 30602}
\email{kazanova@math.uga.edu} 

\begin{abstract}A recent construction of Hacking relates the classification of stable vector bundles on a surface of general type with $p_g = 0$ and the boundary of the moduli space of deformations of the surface. In the present paper we analyze this relation for Godeaux surfaces. We provide a description of certain boundary components of the moduli space of Godeaux surfaces. Also we explicitly construct certain exceptional vector bundles of rank 2 on Godeaux surfaces, stable with respect to the canonical class, and examine the correspondence between the boundary components and such exceptional vector bundles.
\end{abstract}

\maketitle

\vspace{-.26in}

\section{Introduction}

Complex algebraic surfaces $Y$ of  Kodaira dimension 2 are called surfaces of general type. Such surfaces are classified according to  discrete topological invariants $K^2 = c_1(Y)^2$ and $\chi = \chi(\OO_Y)$. Having fixed these invariants,  one can consider the space of all surfaces of general type  with given invariants, i.e., the moduli space $\MM = \MM_{K^2, \chi}$, which itself has the structure of an algebraic variety.

While not compact,  $\MM_{K^2, \chi}$ admits a natural compactification, the moduli space $\Mbar_{K^2, \chi}$ of  stable surfaces introdced by  Koll\'ar  and Shepherd-Barron~\cite{KSB} and Alexeev~\cite{Ale96}. This is an analog of the compactification of Deligne and Mumford $\overline \MM_g$, of  the moduli space $\MM_g$ of smooth curves of genus $g\geqslant 2$~\cite{DM}.

 Such stable surfaces can contain  isolated singularities or even mild singularities along curves.  Also  the moduli space $\Mbar_{K^2, \chi}$ can be arbitrarily singular  \cite{V06},  or have arbitrarily many connected components \cite{C86}.
Fortunately,  some connected components of the  boundary of $\Mbar_{K^2, \chi}$ are relatively well-behaved. They correspond to a degeneration of a smooth surface of general type to a surface with a unique quotient singularity of a special type, first studied by Wahl~\cite{Wahl}.

\begin{defin} A {\it singularity  of Wahl type} $\frac{1}{n^2}(1, na-1)$ is a 
cyclic quotient singularity 
$0\in (\C^2/ \Z/n^2\Z)$
given by 
$$
\Z/n^2\Z \ni 1: (u,v)\mapsto (\zeta u, \zeta^{na-1}v), $$
where $a$ and $n$ are  positive integers such that $a<n$,  $(a,n)=1$, and  $ \zeta = \exp(2\pi i / n^2)$.
\end{defin}

In the absence of local-to-global obstructions, a singularity of Wahl type  on a surface $X$ admits  a smoothing $Y \leadsto X$ such that  $H_2(Y, \Q) \simeq H_2(X, \Q)$.
The local smoothing is determined by a single deformation parameter, so the locus of equisingular deformation of $X$  defines a codimension 1 boundary component of the moduli space of deformations of $Y$.  We will call such a boundary component a $\frac{1}{n^2}(1, na-1)$ boundary component.

In this paper we investigate some of the boundary components  of the moduli space $\Mbar_{1,1}$, containing Godeaux surfaces.
  A Godeaux surface is a minimal surface of general type  with invariants $K^2 =1$ and $p_g=0$. Godeaux surfaces are in some sense the simplest surfaces of general type satisfying $H^1(Y) = H^{2,0}(Y) =0$.

In \cite{Ha11}, Hacking describes a way to construct  exceptional vector bundles of rank $n$  on a smooth surface $Y$  such that $H^1(Y)=H^{2,0}(Y)=0$ using  degenerations $Y \leadsto X$ of  $Y$  to a surface $X$ with a unique singularity of Wahl type $ \frac{1}{n^2}(1,na-1)$.  Exceptional vector bundles on a surface $Y$, discussed in Section~3, are holomorphic vector bundles $E$ such that $\Hom (E, E) =\C$ and $\Ext^1(E, E) = \Ext^2(E, E) =0$.  In particular, such a vector bundle is indecomposable, rigid and unobstructed in families, i.e., it deforms in a unique way in a family of surfaces. 
Exceptional vector bundles have also appeared in decompositions of the derived categories on Godeaux surfaces~\cite{BBS} and on Burniat surfaces~\cite{AO}.

The construction of Hacking gives rise to a correspondence  
\begin{equation}
\label{ch}
\mbox{\{Wahl degenerations\}}  \longrightarrow \mbox{\{exceptional vector bundles}
\}/\sim, 
\end{equation}
where $\sim$ denotes the equivalence relation defined in Remark \ref{equiv}.

This correspondence is  bijective in the case $Y = \Pro^2$~\cite{Ha11}. Thus it was natural to study the correspondence~(\ref{ch}) for other surfaces, in particular surfaces of general type.  
In the present paper we examine the correspondence~(\ref{ch}) for Godeaux surfaces and Wahl degenerations corresponding to the $\frac{1}{4}(1,1)$ singularity.

Throughout the paper we will make use of the classification of Godeaux surfaces  according to  $H_1(Y, \Z)$,  which is cyclic of order at most 5~\cite{Reid}.   There exists a complete description of the moduli space of Godeaux surfaces in each of the cases $H_1(Y, \Z) = \Z/5\Z$, $\Z/4\Z$ or $\Z/3\Z$~\cite{Reid}.   Interestingly, for the remaining two cases, the description of the moduli space of Godeaux surfaces is still unknown. 
Several examples of Godeaux surfaces with  $H_1(Y, \Z)= \Z/2\Z$ were constructed in \cite{Bar84}, \cite{KLP}, \cite{CD89}  and some work towards classification is  done in \cite{C10}. 
In the case $H_1(Y, \Z)=0$ the only known  examples  are described in  \cite{Bar85}, \cite{DW99}, \cite {LP}, but it is not even  known  if these examples belong to the same irreducible component of the moduli space.

\subsection{Results}

First,  we classify all possible degenerations of a smooth Godeaux surface  $Y$ to a surface $X$ with unique singularity of Wahl type $\frac{1}{4}(1,1)$, such that $K_X$ is ample. In other words, we describe the boundary components of the KSBA compactification of the moduli space of smooth Godeaux surfaces corresponding to surfaces with a unique such singularity.

\begin{thm}
\label{intro-d}
The $\frac{1}{4}(1,1)$ boundary components in the the KSBA compactification all parametrize surfaces whose minimal resolution is a proper elliptic surface. There are no such components when $H_1(Y, \Z) = \Z/5\Z$,  and at least one for each other possible value of  $H_1(Y, \Z) $.
\end{thm}
See Theorem~\ref{cases} for a more precise statement.

Second, we classify all exceptional vector bundles of rank 2 on smooth Godeaux surfaces into two orbits under the natural  equivalence relation obtained from the construction of Hacking. We provide the complete description for one of the orbits.

\begin{thm} \label{intro-e} 
 If $E$ is a $K_Y$-stable exceptional vector bundle of rank 2 on a Godeaux surface $Y$ with   $c_1(E)=K_Y$ modulo torsion, then, after tensoring by a torsion line bundle,
 $E$ can be written as an extension
 \begin{equation*}
0\to \OO_Y\to E\to \OO_Y(K_Y+ \sigma)\otimes \II_P\to 0,
\end{equation*}
 where $I_P$ is the ideal sheaf of a point P which is a base point of   $|2K_Y+\sigma|$, and $\sigma \in \Tors Y\setminus 2\Tors Y$, so we must have  $H_1(Y) = \Z/4\Z$ or $\Z/2\Z$.
 
Conversely, given $P$ and $\sigma$ as above, there is a unique non-trivial extension $E$ of this form, and $E$ is a $K_Y$-stable exceptional vector bundle provided $P$ is a simple basepoint.

\end{thm}

 Finally,  we investigate which of these exceptional vector bundles can be obtained using Hacking's construction.

\begin{thm}  \label{intro-c} Let $Y$ be a Godeaux surface with $H_1(Y,\Z)=\Z/4\Z$.
Every $K_Y$-stable exceptional bundle $E$  of rank 2  on $Y$ with $c_1(E) = K_Y$ modulo torsion is equivalent to one arising from a $\frac{1}{4}(1,1)$ Wahl degeneration $Y \leadsto X$ with $K_X$ ample.

\end{thm}

The paper is structured as follows. Section 2 deals with the classification of the degenerations of a Godeaux surface $Y$ to a surface $X$ with a unique singularity of Wahl type $\frac{1}{4}(1,1)$, such that the canonical divisor $K_X$ is ample. A complete classification is provided and some concrete examples are described in detail.

Section 3 contains the analysis of the equivalence classes of certain exceptional vector bundles of rank 2 on smooth Godeaux surfaces $Y$. The classification is provided modulo the equivalence relation arising from the construction of Hacking. A complete description is provided for one of the two equivalence classes. 

Section 4 analyzes the correspondence~(\ref{ch}) in the case $Y$ is a Godeaux surface and $X$ has a unique singularity of the type $\frac{1}{4}(1,1)$.

\section*{Acknowledgments}
I would like to thank my advisor Paul Hacking for suggesting this problem to me, and for his guidance and help. I am grateful to Valery Alexeev, Dustin Cartwright, Stephen Coughlan, Angela Gibney, Daniel Krashen, Julie Rana, and  Jenia Tevelev for many useful discussions.

\section{Wahl Degenerations}

We use the following notation. Let $Y$ be a Godeaux surface, and let $X$ be a $\Q$-Gorenstein degeneration of $Y$ such that $X$  has a unique  singularity, which is of Wahl type $\frac{1}{4}(1,1)$ and $K_X$ is nef.  We denote by $\X$ be the minimal resolution of $X$.  For a point $P$ on $X$ we write $(P \in X)$ to denote a small complex analytic neighborhood of $P \in X$. We use the notation $\Tors (Y) = \Tors H^2(Y, \Z) = \Tors H^2(Y) = \Tors H_1(Y) = H_1(Y)$.

\begin{prop}\label{2.1}
The surface $\X$ is a minimal properly elliptic surface, i.e. $\X$ is minimal of Kodaira dimension 1. 
\end{prop}

\begin{proof}
We start by  computing invariants of the surface $\X$. Notice that since $X$ is a $\Q$--Gorenstein  degeneration of $Y$, we have  $K_X^2=K_Y^2=1$. The exceptional locus $C$  of the minimal resolution of the cyclic quotient singularity of  type $\frac{1}{4}(1,1)$   consists of a single $(-4)$ curve.  

Let $\pi: \X\to X$. Then $K_{\X}= \pi^*K_X-\frac{1}{2}C$ by the  adjunction formula, so $K_{\widetilde X}^2 =0$.

Also  $p_g(\X)= h^2(\OO_{\X})=h^2(\OO_X)$ since $(P\in X)$ is a rational singularity, and $h^2(\OO_{X})=h^2(\OO_Y)=0$ since $(P\in X)$ is a quotient singularity by~\cite{DB81}, 4.6 and 5.3.
Similarly $q(\OO_{\X})= h^1(\OO_{\X})= h^1(\OO_Y)=0$.  Finally,  the fundamental group does not change on a resolution of a rational singularity, so  $\pi_1(\X)=\pi_1(X)$.

So $\widetilde X$ is a surface with  $K_{\widetilde X}^2=0$, $p_g(\widetilde X)=q(\X)= 0$. By Lemma~\ref{lemma:minimal} below $\X$ is a minimal surface. Thus according to the classification of minimal surfaces\cite[p. 244]{BPHV} , the surface $\X$ is either an Enriques surface, a K3 surface, or a properly elliptic surface.   
Note that $\X$ cannot be an Enriques surface or a K3 surface because there exists a curve  $C$ on $\X$ such  that $C\cdot K_{\X} =2$, thus $2K_{\X}$ and $K_{\X}$ are not numerically trivial.  Therefore the surface  $\X$ is  a properly elliptic surface.  
\end{proof}

\begin{lem}\label{lemma:minimal} The surface $\X$ does not contain a $(-1)$-curve. Thus $\X$ is not a rational surface or a surface of general type. 
\end{lem}
\begin{proof} Arguing by contradiction, suppose that there exists a $(-1)$-curve  $F$ on $\X$. 

Denote by   $\bar{F} = \pi(F)$ the image of the curve $F$ under the map $\pi$. 
Recall $K_{\X}= \pi^*(K_X)-\frac{1}{2}C$.   
Since the canonical divisor $K_X$ is big and nef, its intersection index with the curve $\bar F = \pi_*(F)$ on $X$ must be nonnegative. 
But 
 $$K_X\cdot\bar F= K_X\cdot \pi_*( F)= \pi^*(K_X)\cdot F= (K_{\widetilde X}+ \frac{1}{2} C)\cdot F=-1+\frac{1}{2}C\cdot F\geqslant 0,$$ 
thus $C\cdot F\geqslant 2$. 

Now consider the divisor  $D=   C+2F$ on $\X$. We have $D^ 2=  -8+ 4C\cdot F\geqslant 0$ since $C\cdot F\geqslant 2$.
But  $K_{\X}\cdot D =  0$, so $D\in K_{\X}^{\perp}$ and $K_{\X}^2=0$. 
Note that $D$ and $K_{\X}$ are linearly independent in $H^2(\X, \R)$. If $D^2 >$ 0 it is clear that $D$ and $K_{\tilde{X}}$ are linearly independent because $K_{\tilde{X}}^2=0$. If $D^2=0$ then $C \cdot F =2$ and we find $D \cdot F =0$, $K_{\tilde{X}} \cdot F =-1$, so again $D$ and $K_{\tilde{X}}$ are linearly independent.
This contradicts the Hodge Index Theorem. Therefore there is no such curve $F$.\end{proof}

\begin{thm}
\label{cases} The  surface  $\X$ is a properly elliptic surface over $\Pro^1$ with two multiple fibers of multiplicities $m_1, m_2$. In particular $\pi_1(\X)\simeq \Z/(m_1, m_2)\Z$. Write $n:= C\cdot A$, where $A$ is a general fiber of the elliptic fibration and $C$ is the exceptional locus of $\pi: \X\to X$, then we have the following possibilities for $m_1$, $m_2$ and $n$: 
\begin{enumerate}
\item[(a)] $m_1=4$, $m_2=4$, $n=4$;
\item [(b)]$m_1=3$, $m_2=3$, $n=6$;
\item [(c)]$m_1=2$, $m_2=6$, $n=6$;
\item [(d)]$m_1=2$, $m_2=4$, $n=8$;
\item [(e)]$m_1=2$, $m_2=3$, $n=12$.
\end{enumerate}

\end{thm}

\begin{proof}
We have an elliptic fibration
$$
\begin{CD}
\X@<<< \mbox{\{multiple fibers with multiplicities } m_i\}\\
@ VfVV @ VVV\\
B@<<< \{P_i\}
\end{CD}
$$

Denote by $L$ the dual of the line  bundle $R^1f_* \OO_{\X}$ on $B$. By~\cite{FM}, Chapter~I, Lemma~3.18 we have $\deg L = \chi(\OO_{\X})=1$, since $p_g(\widetilde X) = q(\X)=0$.
Moreover, according to~\cite{FM}, Chapter~I, Proposition~3.22, we can compute the genus $g(B)$ of the base curve $B$ using the relation $p_g(\X) = \deg L + g(B)-1 =0$. Thus the base curve $B$ must be isomorphic to $\Pro^1$.

Since the Euler number $e(\X)>0$, we note that by~\cite{FM}, Chapter~II, Theorem~2.3 the fundamental group $\pi_1(\X)$ is isomorphic to the orbifold fundamental group of the base $B\simeq \Pro^1$, i.e.,we have an isomorphism:
$$\pi_1(\X) \simeq \frac{\pi_1(B\setminus\{P_1, \dots, P_r\})}{<\gamma_1^{m_1}, \dots, \gamma_r^{m_r}>},$$
where $\gamma_i$ are loops on the base $B$ of the fibration  about the corresponding points $P_i$.

Now we consider the Kodaira canonical bundle formula (\cite{Friedman}, Theorem~15).
 Let $f: \X\to B$ be a relatively minimal elliptic fibration. Suppose that $F_1, \dots, F_k$ are the multiple fibers of $f$ and that the multiplicity of $F_i$ is $m_i$. Then:
$$ \omega_{\X}= f^*(\omega_B\otimes L)\otimes \OO_{\X}\Big(\sum_i(m_i-1)F_i\Big).$$

 Thus we obtain the following formula for the canonical line bundle of $\X$. 
\begin{equation}
\label{cbf}
K_{\X}=  \Bigg(-1+\sum_i\Big(\frac{m_i-1}{m_i}\Big)\Bigg)A
\end{equation}
in $\Pic(\X)\otimes \Q\simeq H^2(\X, \Q)$,
where $A$ is a general fiber of the fibration. 

Recall that we have a $(-4)$--curve $C$ on $\X$, and that $C\cdot K_{\widetilde X}=2$.

Intersecting~(\ref{cbf}) with $C$ we obtain the equation $2 = (-1 + \sum \frac{m_i-1}{m_i})n$, or 
\begin{equation}
\label{eqn}
1+\frac{2}{n}=\sum_i\Big(\frac{m_i-1}{m_i}\Big).
\end{equation}

Since $n = C\cdot A = (C\cdot F_i)m_i$, we note that $m_i$ divides $n$. Given these conditions, 
 the only integer solutions of the equation~(\ref{eqn}) are: 
\begin{enumerate}
\item[(a)] $m_1=4$, $m_2=4$, $n=4$;
\item [(b)]$m_1=3$, $m_2=3$, $n=6$;
\item [(c)]$m_1=2$, $m_2=6$, $n=6$;
\item [(d)]$m_1=2$, $m_2=4$, $n=8$;
\item [(e)]$m_1=2$, $m_2=3$, $n=12$;
\item[(f)] $m_1=m_2=m_3=2$, $n=4$;
\item[(g)] $m_1=m_2=m_3=m_4=2$, $n=2$.
\end{enumerate}

Thus $\pi_1(\X)$ is Abelian in all cases except the last two. We have $H_1(\X) = \pi_1^{ab}(\X)$ so in the case (f) we can compute $H_1(\X) =\Z/2\Z\oplus \Z/2\Z$, and in the case (g) we have $H_1(\X) = \Z/2\Z\oplus \Z/2\Z\oplus \Z/2\Z$. 
According to~\cite{Ha13}, p.134 we have a surjection   $\phi: H_1(Y)\twoheadrightarrow H_1(X)$ and the kernel of the map $\phi$ is either $\Z/2\Z$ or trivial.  (Here $n=2$ for the singularity of Wahl type $\frac{1}{n^2}(1,na-1)$.)
Thus the last two cases cannot be obtained by a degeneration  of a Godeaux surface $Y$ to a surface $X$.
\end{proof}

\begin{remark} Since we have a map $\phi: H_1(Y)\twoheadrightarrow H_1(X)$, which kernel is either $\Z/2\Z$ or trivial, $H_1(Y) \simeq H_1(X)$ in the cases $H_1(Y) = \Z/5\Z$, $\Z/3\Z$, or trivial. We can have $\Z/2\Z\to H_1(Y)\to H_1(X)$ in the remaining two cases. However, the construction of Hacking requires $H_1(Y) \simeq H_1(X)$, so we will provide specific examples for which this condition is satisfied.
\end{remark}

The following statement is an immediate corollary of Theorem~\ref{cases}. 
\begin{cor} There does not exist a degeneration of a smooth Godeaux surface $Y$ with $H_1(Y)=\Z/5\Z$  to a surface $X$ with a  unique singularity $(P\in X)$ of Wahl type  $\frac{1}{4}(1,1)$, such that $K_{X}$ is ample.
\end{cor}

We provide explicit constructions of degenerations $Y\leadsto X$ when $H_1(Y, \Z)$ is not equal to $\Z/5\Z$.

 \begin{prop} \label{p4}There exists a $\Q$-Gorenstein degeneration of a Godeaux surface $Y$ with $H_1(Y)= \Z/4\Z$  to a surface $X$ with a unique  singularity of type $\frac{1}{4}(1,1)$ such that $K_X$ is ample.
 \end{prop}

\begin{proof}

According to~\cite{Reid} the universal cover $\overline Y$ of a Godeaux surface $Y$ with $H_1(Y)=\Z/4\Z$  is given by a complete intersection of two quartics $\overline Y =\{q_0=q_2=0\} \subset\Pro(1^3,2^2)$ in a weighted projective space with coordinates $x_1, x_2, x_3, y_1, y_3$ of degrees 1, 1, 1, 2, 2, respectively. Then the Godeaux surface $Y$ is the quotient of $\overline Y$ by the $\Z/4\Z$ action generated by $x_i \mapsto \zeta^{i-1} x_{i}$, $y_i \mapsto\zeta^{i-2} y_i$, where $\zeta$ is a primitive 4th root of unity.    We will say that a variable $x_j$ has weight  $ j\in \Z/4\Z$  to indicate that  that $x_j \mapsto \zeta^j x_j$ under the $\Z/4\Z$ action.

We  consider the family $\mathcal P= ( x_1x_3 = x_2^2 + t v_0)\subset \Pro(1^3,2^3) \times \mathbb A^1_t$ with coordinates $x_1, x_2, x_3, v_0, y_1, y_3,t$ which have  weights $0$, $1$, $2$,  $2$, $3$, $1$, $0\in \Z/4\Z$ respectively.
Then if $t\neq 0$, we can solve for $v_0$ and so the general fiber $\mathcal P_t$ is isomorphic to $\Pro(1^3, 2^2)$.

The special fiber $\mathcal P_0$ is isomorphic to the projective space $\Pro(1^2,4^3)$ with coordinates $u_0, u_1, v_0, y_1, y_3$ with weights  $0$, $1$,  $2$, $3$, $1\in  \Z/4\Z$ via setting $x_1 = u_0^2$, $x_2 = u_0u_1$, $x_3=u_1^2$. Then $\overline Y\subset \mathcal  P_t$ degenerates to a complete intersection $\overline X\subset \Pro(1^2, 4^3)$, given by two equations of degree 8 and weights $0$ and $ 2$. 

We describe an example of a $\Z/4\Z$ invariant quasismooth complete intersection in $\Pro(1,1,4,4,4)$. We define  $\overline X=\{f_0=f_2=0\}$, where
\begin{equation}
\begin{aligned}
f_0&=u_0^8+u_1^8+u_0^4u_1^4+y_1y_3+v_0^2+v_0u_0^2u_1^2;\\
f_2&=u_0^6u_1^2+u_0^2u_1^6+y_1^2+y_3^2+u_0^4v_0+u_1^4v_0.
\end{aligned}
\end{equation}

Now $\overline X$ meets the locus $(u_0=u_1=0)\subset \Pro(1^2,4^3)$ transversely at exactly four points $(0, 0, 1, \pm \zeta, \mp\zeta)\in \Pro(1^2,4^3)$, where $\zeta^4 = -1$.
Thus $ \overline X$ has four $\frac{1}{4}(1,1)$ singularities and no other singularities, and its quotient $X=\overline X\big/(\Z/4\Z)$ has a unique $\frac{1}{4}(1,1)$ singularity.

Using the adjunction formula to compute the canonical divisor of $X$ we obtain  $K_{\overline X}= (-14 H + 8 H + 8 H )|_{\overline X} = 2 H|_{\overline X}$, where $H$ is a general hyperplane divisor on $\Pro(1^2, 4^3)$,
so $K_{X}= p_* K_{\overline X}$ is ample. 
\end{proof}

In the proof of Proposition~\ref{p4}, we constructed degeneration on the ambient weighted projective space. In contrast, we will now construct a degeneration in the $H_1(Y, \Z) = \Z/3\Z$ by choosing the equations which meet the singular locus of the weighted projective space.

\begin{prop} \label{p3}There exists a $\Q$-Gorenstein degeneration of a Godeaux surface $Y$ with $H_1(Y)= \Z/3\Z$ to a surface $X$ with a unique  singularity of type $\frac{1}{4}(1,1)$, such that $K_X$ is ample.
\end{prop}

\begin{proof} In the weighted projective space $\Pro(1^3,2^3,3^3)$ with coordinates $x_i$, $y_i$, $z_i$, where   $i \in \{0,1,2\}$,  consider equations
\begin{equation}
\label{Z3}
\begin{aligned}
&r_2x_1x_0-x_2z_2+y_1y_0=0\\
&x_1z_1-x_2r_1x_0-y_0y_2=0\\
& x_1(Sx_0-r_0y_0)-y_1r_1x_0-z_2y_2=0\\
&x_2(Sx_0-r_0y_0)- y_1z_1-y_2r_2x_0=0\\
&y_0(Sx_0-r_0y_0)-z_1z_2+r_1r_2x_0^2=0\\
&x_0z_0-y_1y_2-r_0x_1x_2=0\\
&y_0z_0-Sx_1x_2+r_2x_1y_2+r_1x_2y_1=0\\
&z_0z_1-Sx_2y_2-r_0r_1x_2^2+r_2y_2^2=0\\
 &z_0z_2-Sx_1y_1-r_0r_2x_1^2+r_1y_1^2=0
\end{aligned}
\end{equation}
Miles Reid showed in \cite{Miles2}  that any Godeaux surface $Y$ with $H_1(Y)= \Z/3\Z$ can be obtained by setting in (\ref{Z3}) $x_0+x_1+x_2=0$, $z_0+z_1+z_2=0$, and $r_i=$ quadratic, $S=$ cubic expression in $x_i$, $y_i$, with the $\Z/3\Z$-action given by cyclic permutation of $(0,1,2)$.  See the example 7.1 in \cite{Miles2}  for details. (See also \cite{Reid}, section 3.) Moreover, denote by $W\subset \Pro(1^3, 2^3, 3^3)$  the $\Z/3\Z$ cover of $Y$, then $K_W=H|_W$, where $H$ is a general hyperplane on $ \Pro(1^3, 2^3, 3^3)$.

We consider a one parameter family of Godeaux surfaces given by setting in the equations~(\ref{Z3})
$$
\begin{aligned}
&r_0=ty_0+ y_1+2y_2+x_1^2\\
&r_1=ty_1+ y_2+2y_0+x_2^2\\
&r_2=ty_2+  y_0+2y_1+x_0^2\\
& S= x_0^3+x_1^3+x_2^3,
\end{aligned}
$$
where $t\in \C^1$, along with $x_0+x_1+x_2=0$ and $z_0+z_1+z_2=0$. Then for each fixed small $t\neq 0$ after taking a quotient  by the group action we obtain a smooth Godeaux surface with $H_1= \Z/3\Z$.

Using Macaulay2, we can see that at $t=0$  the surface $Z$ defined by these equations has three $\frac{1}{4}(1,1)$ singularities at the points where $x_i=z_i=0$, $y_j=1$, $y_k=0$ for all $i$ and $j\neq k\in \{0,1,2\}$ and no other singularities. Direct calculation shows that all three of these singularities are $\frac{1}{4}(1,1)$ singularities. Thus the quotient by the  $\Z/3\Z$ action is  a surface with a unique singularity of type $\frac{1}{4}(1,1)$  which is a degeneration of a smooth Godeaux  surface. Also $K_Z= H|_Z$, so $K_X$ is ample.
\end{proof}

	A construction of a degeneration $X$ of a Godeaux surface $Y$ with a unique singularity of type $\frac{1}{4}(1,1)$  in the cases $H_1(Y)=\Z/2\Z$ and $H_1(Y)=0$   arises from the theory of $\Q$-Gorenstein smoothing for projective surfaces
with special quotient singularities.

\begin{prop} \label{p0}Let $Y$ be a simply connected Godeaux surface, whose construction is described in the ~\cite{LP}, Section 7, Construction A2. Then there is a degeneration $Y\leadsto X$, where $X$ has a unique singularity of type $\frac{1}{4}(1,1)$, and $K_X$ is nef.   Let $\X$ be a minimal resolution of $X$, then $\X$ is a Dolgachev surface, i.e. a properly elliptic surface over $\Pro^1$ with exactly  two multiple fibers of multiplicities 2 and 3.
\end{prop}

\begin{proof}

The surface $X$ is obtained by smoothing all but one $\frac{1}{4}(1,1)$ singularity in the surface $X'$ in the construction  of  Lee and Park. Lee and Park give an example of a simply connected Godeaux surface using $\Q$--Gorenstein smoothing theory~\cite{LP}. They first construct a surface $X'$ having one cyclic quotient singularity of  type $\frac{1}{36}(1, 5)$ and two cyclic quotient singularities of each of the types $\frac{1}{4}(1,1)$ and $\frac{1}{16}(1,3)$ such that $K_{X'}$ is nef. They prove that $X'$ has a $\Q$--Gorenstein smoothing such that a general fiber of the family  is a simply connected, minimal, complex surface of general type with $p_g = 0$ and $K^2 = 1$. See~\cite{LP}, Section 7, Construction A2  for details on the construction, as well as the outline of the proof. 

In particular, since $H^2(T_{X'})=0$ by~\cite{LP}, the cyclic quotient singularities can be smoothed independently. 
So there exists a $\Q$-Gorenstein smoothing such that the deformation of one $\frac{1}{4}(1,1)$ singularity is trivial and the remaining singularities are smoothed.

Then the resulting surface $X$ is a $\Q$-Gorenstein degeneration of the Godeaux surface having a unique $\frac{1}{4}(1,1)$ singularity. 
\end{proof}

\begin{prop} \label{KLP-cor}Let $Y$ be a Godeaux surface with $H_1(Y) = \Z/2\Z$, whose construction is described in \cite{KLP}, Section 3, Example 1. Then there is a degeneration $Y\leadsto X$, where $X$ has a unique singularity of type $\frac{1}{4}(1,1)$, and $K_X$ is ample. The surface $X$ is obtained by smoothing all but one $\frac{1}{4}(1,1)$ singularity in the surface $X'$ in the construction  of Keum, Lee and Park. 
\end{prop}

\begin{proof}
A surface $X$ with $H_1(X) = \Z/2\Z$ and such that its minimal resolution is an elliptic fibration with two multiple fibers of multiplicities $2$ and $4$ can be obtained by smoothing all but one $\frac{1}{4}(1,1)$ singularity in~\cite{KLP}, Example 3.1.  

Keum, Lee and Park  obtain  a surface $X'$ which 
 has two cyclic quotient singularities of the type $\frac{1}{27}(1, 8)$ and two cyclic quotient singularities of type $\frac{1}{4}(1,1)$ such that $K_{X'}$ is ample. They prove that $X'$ has a $\Q$--Gorenstein smoothing such that a general fiber of the family is a minimal, complex surface of general type with $p_g = 0$, $K^2 = 1$, and $H_1= \Z/2\Z$. See \cite{KLP}, Section 3, Example 1 for details on the construction, as well as the outline of the proof. 

In particular, since $H^2(T_{X'})=0$~\cite{KLP}, the cyclic quotient singularities can be smoothed independently. 
So there exists a $\Q$-Gorenstein deformation such that the deformation of one $\frac{1}{4}(1,1)$ singularity is trivial and the remaining singularities are smoothed.
\end{proof}

\begin{remark}  Proposition \ref{prop-d-2div} together with Proposition \ref{c-2-div} imply that this degeneration $X$ corresponds to the case when the minimal resolution $\X$ has two multiple fibers of multiplicity $2$ and $4$. \end{remark}

\begin{remark} At the moment we are missing the construction of a surface $X$  in case $(c)$, i.e. such that the resolution $\X$ has two multiple fibers of multiplicities $2$ and $6$. Note that in this case the canonical line bundle $K_{\X}$ will be 2--divisible modulo torsion. We expect this degeneration to exist, and the construction of this degeneration would require an argument similar to~\cite{LP} or~\cite{KLP}.  \end{remark}

\begin{proof}[Proof of the Theorem \ref{intro-d}.]
By Proposition \ref{2.1} and Theorem \ref{cases}  the minimal resolution $\X$ of $X$ is a minimal elliptic surface with exactly two multiple fibers. In particular, there does not exist a degeneration $Y\leadsto X$ if $H_1(Y, \Z) = \Z/5\Z$. Propositions \ref{p4}, \ref{p3}, \ref{p0}, \ref{KLP-cor} provide the explicit constructions of degenerations in all other cases. Thus there is at least one $\frac{1}{4}(1,1)$ boundary component for each other $H_1(Y, \Z)$. 
\end{proof}

\begin{prop} \label{prop-d-2div}Let $Y\leadsto X$ be the degeneration of a smooth Godeaux surface $Y$ with $H_1(Y) = \Z/2\Z$ to a surface $X$ with a unique Wahl singularity of type $\frac{1}{4}(1,1)$ obtained in the Proposition~\ref{KLP-cor}. Then the divisor $K_X+\sigma$, where $0\neq \sigma \in \Tors H_2(X)$, is not 2-divisible in $H_2(X, \Z)$.
\end{prop}

\begin{proof} 
Let $\bar Y$ be the Enriques surface. Following~\cite{KLP}, we blow the surface $\bar Y$ up at five points to obtain the surface $\wZ $, such that the rank of the class group $\Cl \wZ $ is equal to 15. Let  $\pi: \wZ\to X'$ be the map contracting the four chains of $\Pro^1$'s.

We would like show that $K_X+ \sigma$ is not 2-divisible in $H_2(X,\Z)$.

Note that it is  sufficient to show that $K_{X'}$ is not 2-divisible in $H_2(X', \Z)/\Tors X'$. Indeed, we have the specialization map
$sp: H_2(X, \Z) \to H_2(X', \Z)$, such that $sp(K_X) = K_{X'}$.

Moreover, it suffices to show that $K_{X'}$ is not 2-divisible in $H_2(X', \Z)/\Tors X'$, since there is a surjective map $H_2(X',\Z)\twoheadrightarrow H_2(X', \Z)/\Tors X'$, mapping $(K_{X'}+\sigma)\mapsto [K_X']$.   So we will consider everything up to a torsion element. Write $L = \Cl \wZ/\Tors \wZ$.
On $\wZ$ denote the curves in the fiber by $C_1, \dots, C_{14}$, and the two bisections by $S_1, S_2$ as  shown in Figure 1.

\begin{figure}
\caption{Special fiber on $Z$.}
\label{Fig1}
  \centering
\includegraphics[width=0.5\textwidth]{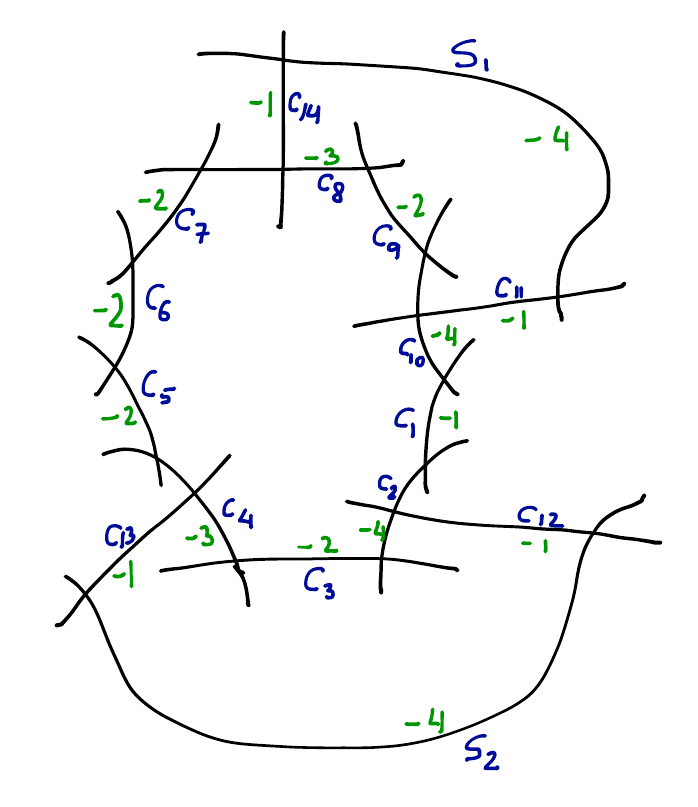}
\end{figure}

Note that $C_1, \dots, C_{14}$ are linearly independent in $L$ since they form a fiber of the elliptic fibration, and denote by $A= \langle C_1, \dots, C_{14}, S_1\rangle$. Then $A$ has rank 15, so it has  finite index in $L$. Using the intersection matrix for $A$ we compute that the index $|L/A|=6$.

Let $M = \langle C_2, \dots, C_5, C_7, \dots, C_{10}, S_1, S_2\rangle$ be the set of all $\pi$-exceptional divisors on $\wZ$, then we need to show that $K_{X'}$ is not 2-divisible in $\Cl(X')/\pi_*(\Tors\wZ) = L/M$ 

Let $G = (1/2) F$, where  $F$ is a general fiber of the fibration $\widetilde Z\to \Pro^1$. Note that $G$ is an element of $L$ since $\wZ$ has two  multiple fibers of multiplicity 2 because it is a blowup of an  Enriques surface \cite{BPHV}, Chapter VIII, Lemma 17.1.

 Consider $N= A+ \Z\cdot G$, then we have $A\subset N\subset L$.  Note that the containments are strict as we have $G \not\in A$. To show this, note that $G= (1/2) F\in F^{\perp}$ since $F^2 =0$. 
  But $G\not\in A\cap F^{\perp}= \langle C_1, \dots, C_{14}\rangle$, since $G = \frac{1}{2}(2C_1+ C_2 + \dots + C_{14})$.

Thus we have $|N/A|=2$, so $|L/ N|=3$.   Moreover, a basis of $N$ is given by $\langle C_1, \dots, C_{13}, G, S_1\rangle$.

 Unfortunately $M\not\subset N$, but since $|L/N|=3$, we have $N\otimes \Z/2\Z\simeq L \otimes \Z/2\Z$.
   So it is enough to show that $K_{X'}$ is nonzero in $(N\otimes \Z/2\Z)/ (M\otimes \Z/2\Z)$ to conclude that it is not 2-divisible in $L/M$.

On $\wZ$ we have $S_i\cdot F=2$,  $i = 1,2$, and $C_i \in F^{\perp}$. 
Note that $S_1-S_2\in F^{\perp}$, so we can write $S_1-S_2 = \sum _{i=1}^{14}x_iC_i$ for some $x_i \in \Q$. Using  the intersection matrix  $(C_i\cdot C_j)$ we can compute the vector 

$\mathbf x = (x_1, \dots , x_{14})  =\frac{-1}{3}(14, 5, 4, 3, 5, 7, 9, 11, 10, 9, 12, 2, 14, 0) + \lambda/6 (2, 1, \dots, 1)$ for some $\lambda\in \Z$, since  we can write the general fiber $F = 2G=\sum_{i=1}^{14}m_i C_i$, where $m_1=2$ and $m_i = 1$ for all $i\neq 2$. 

Thus modulo 2 we have $S_1-S_2 = C_2+ C_4+C_5+C_6+C_7+C_8+C_{10} + \mu G\in N$, for some $\mu \in \Z$.

So $$\frac{N\otimes \Z/2\Z}{M\otimes \Z/2\Z }= \frac{\langle C_1, \dots, C_{13}, G, S_1\rangle\otimes\Z/2\Z }{ \langle C_2, \dots, C_5, C_7, \dots, C_{10}, S_1,S_1- S_2\rangle\otimes\Z/2\Z} = $$
$$=\frac{\langle C_1, C_6, C_{11}, C_{12}, C_{13}, G\rangle}{\langle C_6+ \mu G\rangle}\otimes \Z/2\Z.$$

We have $K_{\widetilde Z} = p^* K_{\hat Y} + (C_1 + C_{11}+ C_{12}+ C_{13}+ C_{14}) = C_1 + C_{11}+ C_{12}+ C_{13}+ C_{14}$ in $L$.  Then $K_{X'} = \pi _* K_{\widetilde Z}\in \Cl(X') /\pi_*\Tors \widetilde Z= L/M$, so $K_{X'} = C_1 + C_{11}+ C_{12}+ C_{13}+ C_{14}$. 

Finally in $(N\otimes \Z/2\Z)/(M\otimes \Z/2\Z )\simeq (\Z/2\Z)^6/\langle0,1,0,0,0,\mu\rangle$ we have $K_{X'}\equiv 2G-C_1-C_6$, so $K_{X'}\equiv C_1+C_6$. Thus $K_{X'}$ is nonzero in $(N\otimes \Z/2\Z)/(M\otimes \Z/2\Z)$, so it is not divisible in $H_2(X', \Z)$.\end{proof}


\section{Exceptional Vector Bundles}

In this section we study $K_Y$ stable exceptional vector bundles of rank 2 on smooth Godeaux surfaces $Y$.

\begin{defin}
An {\it exceptional vector bundle}  $E$ on a surface $Y$ is a holomorphic vector bundle such that $\Hom (E, E) =\C$ and $\Ext^1(E, E) = \Ext^2(E, E) =0$.  
\end{defin}

\begin{defin} The {\it slope} of a vector bundle $E$ of rank $r$   on a surface $Y$  with respect to an ample line bundle $H$ is 
\begin{equation}
\mu(E)= \frac{c_1(E)\cdot H}{r}.
\end{equation}
\end{defin}

\begin{defin} A vector bundle $E$ on a surface $Y$  is called {\it stable with respect to an ample divisor}  $H$ if for every  vector bundle $F\hookrightarrow E$, such that $0<\rank(F)<\rank(E)$  we have $\mu(F)<\mu(E)$. 
\end{defin}

\begin{defin} Let $E$ and $F$ be two vector bundles on a smooth projective surface $Y$. Define
$$\chi(E, F) =\sum_{i=0}^2(-1)^i\Ext^i(E,F).$$ 
\end{defin}

We will use the following two facts. 
\begin{lem} \label{intro:thm:exc}Let $E$ be a vector bundle on a smooth projective surface $Y$. Then $E$ is exceptional if and only if the dual vector bundle $E^{\vee}$ is exceptional. Moreover, for any line bundle $L$ on $Y$ we have $E$ is exceptional if and only if $E\otimes L$ is exceptional.
\end{lem}
\begin{proof} We have $\Ext^i(E, E)= H^i(\cHom (E,E))$, as well as $\cHom(E\otimes L, E\otimes L) = \cHom (E,E)$, also $\cHom(E^{\vee}, E^{\vee})= \cHom (E, E)$.\end{proof}

\begin{lem} \label{thm:c2} Let $Y$ be a smooth surface with $\chi(\OO_Y)=1$.
Let $E$ be an exceptional vector bundle of rank $n$ on $Y$.  Then 
\begin{equation}\label{c2}
c_2(E) = \frac{n-1}{2n}(c_1(E)^2+n+1)
\end{equation}
\end{lem}
\begin{proof} On one hand we have $\chi(E, E) = \Hom(E, E) - \Ext^1(E, E) + \Ext^2(E, E) = 1$, since $E$ is exceptional. Also by the Hirzebruch--Riemann--Roch formula we have 
$$1=\chi(\mathcal End E) =n^2\chi(\OO_Y)+ (n-1)c_1(E)^2 -2nc_2(E).$$
We can get the formula~(\ref{c2}) by solving this equation for $c_2(E)$.
\end{proof}

In \cite{Ha11} Hacking describes a way to construct exceptional vector bundles of rank $n$ on a
smooth surface $Y$ such that  $H^1(\OO_Y)=H^2(\OO_Y)=0$ using degenerations $Y\leadsto X$ of $Y$ to a surface $X$ with
a unique singularity of Wahl type $\frac{1}{n^2}(1, na - 1)$. We quote the following theorem of Hacking.

\begin{thm}[\cite{Ha11}, Thm. 1.1]\label{thm:hacking} Let $\mathcal X/(0\in S)$ be a one parameter $\Q$-Goren\-stein smoothing of a normal projective surface $X$ with a unique singularity $(P\in X)$ of Wahl type $\frac{1}{n^2}(1, na-1)$. Let $Y$ denote a general fiber of $\mathcal X/(0 \in S)$. Assume that $H^1(\OO_Y)=H^2(\OO_Y)=0$ and the map $H_1(Y, \Z)\to H_1(X, \Z)$ is injective.
Then after a finite base change $S'\to S$ there is a rank $n$ reflexive sheaf $\mathcal E$ on $\mathcal X'$ such that  \item $E:=\mathcal E|_Y$ is an exceptional bundle on $Y$.
The topological invariants of $E$ are given by:  $\rank(E)=n$,  
${c_1(E)\cdot K_Y=\pm a \mod n}$, and $c_2(E) = \frac{n-1}{2n}(c_1(E)^2+n+1)$. 

If $\mathcal H$ is an ample line bundle on $\mathcal X$ over $S$, then $E$ is slope stable with respect to $\mathcal H|_Y$. 

\end{thm}

 For a $K_Y$-stable exceptional vector bundle $E$ as produced by Theorem~\ref{thm:hacking}, we will be most interested in its slope vector, which is the numerical invariant $c_1(E)/\rank(E)\in H^2(Y, \Q)$. Furthermore, we will say that two slope vectors are equivalent $``\sim"$ if they differ a combination of  translation by $H^2(Y, \Z)$, multiplication by $\pm1$, and the action of the monodromy group $\Gamma \subset \Aut(H^2(Y, \Z))$.  (The first two correspond to the vector bundle operations $E\leadsto E\otimes L$ for a line bundle $L$, and $E\leadsto E^*$.)

In this section we study  exceptional vector bundles $E$ of rank $2$ on a  Godeaux surface $Y$ such that $E$ is stable with respect to $K_Y$ up to the equivalence relation defined in Remark~\ref{equiv}. Note that these exceptional vector bundles correspond to the degenerations of smooth Godeaux surfaces to surfaces with a unique singularity of type $\frac{1}{4}(1,1)$ in the construction of Hacking.

\begin{remark}\label{equiv} Since we are only dealing with vector bundles of the same rank,
we can redefine the equivalence relation $``\sim"$ from Theorem~\ref{thm:hacking} as follows. Two exceptional vector bundles $E_1$ and $E_2$ of the same rank are equivalent if $c_1(E_1)$ can be obtained from $c_1(E_2)$ by translation by $\rank(E)H^2(Y, \Z)$, multiplication by $\pm 1$ and the action of the monodromy group. So we will be thinking of the equivalence relation $\sim$ as a relation on $c_1(E)\in H^2(Y, \Z)/\Tors Y$.
\end{remark}

Then  we can compute orbits of $c_1(E)$ under the equivalence relation $\sim$. 

\medskip

The lattice $B:= H^2(Y, \Z)/\Tors Y$of a Godeaux surface $Y$  is isomorphic to $ \Z K_Y\oplus (-E_8)$.
We can translate $E$ by a line bundle $\OO(M)$ for some 
$M\in H^2(Y, \Z)$ so that   $c_1(E)$ can be replaced by $c_1(E\otimes \OO(M)) = c_1(E) + 2c_1(\OO(M))$. Thus  $c_1(E)\in H^2(Y, \Z) \mod \Tors Y$ can be considered as $c_1(E)\in B/2B\simeq (\Z/2\Z)^9$.

The monodromy group here is $$\Gamma \subseteq \Aut(B, K_Y) = \{g \in \Aut B\mid g(K_Y)=K_Y\}.$$ 
We have $\Aut (B, K_Y) = \Aut(-E_8)$, and the automorphism group $\Aut (-E_8)$ is equal to the Weyl group $W(E_8)$.
We expect the monodromy group  to be equal to  $W(E_8)$. Under this assumption the following lemma holds.

\begin{lem} 
\label{lem:c1}
Suppose that the monodromy group is equal to $W (E_8)$. Up to the equivalence relation generated by translation by an element of $H^2(Y, \Z)$,  and the monodromy group, there are two equivalence classes  of $c_1(E)\in B$ with $c_1(E)^2 \equiv 1\mod 4$, one is given by $c_1(E) \sim K_Y$ modulo torsion. 
\end{lem}
\begin{proof}

Note that $(H^2(Y, \Z)/\Tors Y, \cup) = (H^2(\operatorname{Bl}^8 \Pro^2), \cup)$, so instead of computing the orbits of the action of the monodromy group on $(H^2(Y, \Z)/\Tors Y, \cup)$, we can find them for $(H^2(\operatorname{Bl}^8 \Pro^2), \cup)$. The lattice $(H^2(\operatorname{Bl}^8 \Pro^2), \cup)$ has  basis $\pi^*(H)$, $E_1, \dots, E_8$, where $\pi: \operatorname{Bl}^8 \Pro^2\to \Pro^2$,  $H \subset \Pro^2$ is a hyperplane, and $E_1, \dots, E_8$ are exceptional divisors on $\operatorname{Bl}^8\Pro^2$ satisfying $E_i^2 =-1$ and $E_i \cdot E_j = 0$ for $i\neq j$.  Then for the $\Aut(E_8)$ group action on the lattice $H^2(\operatorname{Bl}^8 \Pro^2))/2H^2(\operatorname{Bl}^8 \Pro^2))$ there are only 2 orbits (given by $\pi^*(H)$ and $-3\pi^*(H)  +  E_1 + \dots + E_8 = K_{\operatorname{Bl}^9 \Pro^2}$).  Thus there are only two corresponding orbits in $(H^2(Y, \Z), \cup)$. \end{proof}

We use the following construction of  rank 2 vector bundles.

\begin{thm}[\cite{Hl}, Theorem 5.1.1] 
\label{Hlt}
Let $Z\subset X$ be a local complete intersection of codimension two, and let $M$ and $L$ be line bundles on $X$. Then there exists an extension
$$0\to L\to E\to M\otimes \II_Z\to 0$$ 
such that $E$ is locally free if and only if the pair $(L^{\vee}\otimes M\otimes K_X, Z)$ has the Cayley--Bacharach property:

(CB) If $Z'\subset Z$ is a subscheme with $\operatorname{length}
(Z') = \operatorname{length}
(Z)-1$ and $s\in H^0(X,   L^{\vee}\otimes M\otimes K_X)$ is a section with $s|_{Z'}=0$, then $s|_Z=0$.

Moreover, $c_1(E)=c_1(L)+c_1(M)$, and  $c_2(E)= c_1(L)\cdot c_1(M)+\length (Z)$.
\end{thm}

After a twist by a line bundle ($E\mapsto M\otimes E$) there exists a section $s\in \Gamma(Y, E)$ with isolated zeroes, so any rank two vector bundle on a surface $Y$  can be described as an extension
\begin{equation}
\label{extdef}
0\to \OO_Y\to E\to L \otimes \II_Z\to 0,
\end{equation}
where $L=c_1(E)$ is a line bundle on the surface $Y$, $Z$ is a zero dimensional subscheme of $Y$ with $\operatorname{length}(Z)=c_2(E)$, and  the divisor $K_Y+L$ satisfies the Cayley-Bacharach property with respect to $Z$.

The following technical lemma gives us a method to check if the vector bundle constructed using the exact sequence~(\ref{extdef}) is exceptional.

\begin{lem}
\label{conditions}
 Let $Y$ be a Godeaux surface with $K_Y$ ample. Suppose that  a vector bundle $E$ of rank 2 on $Y$ and such that $c_2(E) = \frac{1}{4}(c_1(E)^2+3)$ is given by an exact sequence
\begin{equation}
\label{generalE}
0\to \OO_Y \to E \to L \otimes \II_Z\to 0.
\end{equation}
Assume that $L\cdot K_Y>0$. Then $E$ is exceptional if $Z\neq\emptyset$, $L\neq K_Y$, $H^0(L\otimes \II_Z) = 0$, and  $H^0( \OO_Y(K_Y)\otimes L \otimes \II_Z^2)=0$.
\end{lem}

\begin{proof}

By the Riemann--Roch theorem and our assumption that  $c_2(E) = \frac{1}{4}(c_1(E)^2+3)$ we have $\chi(\End E)=1$, see Theorem~\ref{thm:c2}. Thus it is enough to prove that $\Hom(E,E)=\C$ and $\Ext^2(E,E)=0$.    By the decomposition $H^2(Y, \Z)/\Tors =\langle K_Y\rangle\oplus\langle K_Y^{\perp}\rangle$  we can write $L = mK_Y + A$, where $A\in K_Y^{\perp}$. Here $m = L\cdot K_Y$ and we assume $m>0$.

First let us examine the condition $\Hom(E,E)=\C$. By applying $\Hom(-, E)$ to~(\ref{generalE}) we obtain the exact sequence
$$0\to H^0(E\otimes L^{\vee})\to \Hom(E,E)\to H^0(E)\to \Ext^1(L\otimes \II_Z, E)\to\dots .$$
By tensoring~(\ref{generalE}) with $L^{\vee}$ we get
$$0\to L^{\vee}\to E\otimes L^{\vee}\to \II_Z\to 0.$$

Note that $H^0(L^{\vee})=0$ since $K_Y$ is ample,  and $L\cdot K_Y>0$. Also $H^0(\II_Z)=0$ because $Z\neq \emptyset$ and $H^0(\OO_Y)=\C$.  Therefore $H^0(E\otimes L ^{\vee})=0$.
By~(\ref{generalE}) $H^0(E)/H^0(\OO_Y)\simeq H^0(L\otimes \II_Z) $, using  $H^1(\OO_Y)=0$.
So the condition $\Hom(E,E)=\C$ is implied by our assumption
\begin{equation}
\label{condition1}
H^0(L\otimes \II_Z) = 0.
\end{equation}

\medskip

Now consider the condition $\Ext^2(E,E)=0$. Write $\omega_Y = \OO_Y(K_Y)$.
Note that $\Ext^2(E,E)=\Hom(E, E\otimes \omega_Y)^*$ by Serre duality, so it suffices to prove that $\Hom(E, E\otimes \omega_Y)$ is trivial.

Tensoring the exact sequence ~(\ref{generalE}) with $\omega_Y$, we obtain the following exact sequence
\begin{equation}
\label{EK+a}
0\to\omega_Y\to E\otimes \omega_Y\to\omega_Y\otimes L \otimes \II_Z\to 0.
\end{equation}
Now we can apply $\mathcal Hom (E, -)$ to the exact sequence~(\ref{EK+a}) using the fact  that $\mathcal Hom(E,-)=E^{\vee}\otimes -  = E\otimes L^{\vee}\otimes -$. Here $E^{\vee}\simeq E\otimes L^{\vee}$, since $\rank E=2$, and $L = \det E = \wedge^2 E$, so the perfect pairing $E\otimes E\to \wedge^2E=L$ induces an isomorphism $E\simeq E^{\vee}\otimes L$.
The result will be an exact sequence again since $E$ is locally free:
\begin{equation*}
0\to E\otimes L^{\vee} \otimes \omega_Y\to \mathcal Hom(E, E\otimes \omega_Y)\to E\otimes\omega_Y\otimes \II_Z\to 0.
\end{equation*}

By tensoring the exact sequence~(\ref{generalE}) with $L^{\vee}\otimes \omega_Y$ we obtain that  $$H^0(E\otimes L^{\vee}\otimes \omega_Y)=0,$$ 
because  $H^0(\OO_Y(-(m-1)K_Y - A)) =0 $ for all $m>1$, and  $H^0(\OO_Y(K_Y))=0$. In the case $m=1$ we have $H^0(\OO_Y(-A))=0$ unless $A=0$, because $A\cdot K_Y=0$, and $K_Y$ is ample. By Lemma~\ref{noexK} below, $E$ is not exceptional in the remaining case $A=0$, i.e. $L= K_Y$.

So it suffices to show that  $H^0(E\otimes\omega_Y\otimes \II_Z)=0$. 
We will show that $$H^0(E\otimes \omega_Y\otimes \II_Z)\simeq H^0(\omega_Y\otimes L \otimes \II_Z^2).$$

By tensoring $(E\otimes\omega_Y)$  with the exact sequence
$0\to\II_Z\to \OO_Y\to \OO_Z\to 0$
we obtain the exact sequence
$$0\to E\otimes\omega_Y \otimes \II_Z\to  E\otimes\omega_Y \to E\otimes\omega_Y|_Z\to 0$$
and the associated long exact sequence of cohomology
\begin{equation*}
0\to H^0(E\otimes \omega_Y\otimes \II_Z)\to H^0(E\otimes \omega_Y)\to H^0((E\otimes\omega_Y)|_Z)\to\dots
\end{equation*}
Now using~(\ref{EK+a})  we get $H^0(E\otimes\omega_Y)\simeq H^0(\omega_Y\otimes L \otimes \II_Z)$ since $H^0(\omega_Y)=H^1(\omega_Y)=0$.

Also we can restrict the sequence~(\ref{EK+a}) to $Z$ to obtain an exact sequence
\begin{equation*}
\omega_Y|_Z\to E\otimes \omega_Y|_Z\to  \omega_Y\otimes L \otimes \II_Z|_Z\to 0.
\end{equation*}
The map $s:\omega_Y\to E\otimes\omega_Y$ is defined by a section of $E$ which vanishes on $Z$, so the map $\omega_Y|_Z\to E\otimes\omega_Y|_Z$ is the zero map, thus $H^0(E\otimes \omega_Y|_Z)\simeq H^0(\omega_Y\otimes L \otimes \II_Z|_Z)$.

By tensoring the exact sequence $0\to \II_Z\to \OO_Y\to \OO_Z\to 0$ with $\II_Z$, we obtain $\II_Z\otimes \II_Z\to \II_Z\to\II_Z|_Z\to 0$, or 
\begin{equation}\label{IZ}
0\to \II_Z^2\to   \II_Z\to\II_Z|_Z\to 0
\end{equation}

We have the following commutative diagram with  exact rows
$$
\begin{CD}
0\rightarrow H^0(E\otimes \omega_Y\otimes \II_Z)@>>> H^0(E\otimes \omega_Y)@>>>H^0((E\otimes \omega_Y)|_Z)\\
@VaVV@VbVV@VcVV\\
0\rightarrow H^0(L\otimes \omega_Y\otimes \II_Z^2)@>>> H^0(L\otimes \omega_Y\otimes \II_Z)@>>>H^0((L\otimes\omega_Y\otimes \II_Z)|_Z),\\
\end{CD}
$$
where the bottom row is obtained by tensoring (\ref{IZ}) with $(L\otimes \omega_Y)$, and then taking global sections.
Since the maps $b$ and $c$ are isomorphisms, we conclude that $a$ is an isomorphism as well.

Thus $\Ext^2(E, E) =0$ if \begin{equation}
\label{condition2}
 H^0( \omega_Y\otimes L \otimes \II_Z^2)=0.
\end{equation}\end{proof}

Our analysis of exceptional vector bundles on a Godeaux surface will be divided into two cases corresponding to the two orbits from Lemma\ref{lem:c1}. We first handle the case when $c_1(E)\sim K_Y \mod \Tors Y$ by the following theorem.
We state more precise version of the Theorem~\ref{intro-e} here.

\begin{thm}
\label{Kvb}
 Let $Y$ be a Godeaux surface with $K_Y$ ample, $\sigma \in \Tors Y$ and $P$ a base point of $|2K_Y+\sigma|$. There is a unique  
 vector bundle $E$ of rank 2 on $Y$ given by an extension
\begin{equation}
\label{K}
0\to \OO_Y\to E\to \OO_Y(K_Y+ \sigma)\otimes \II_P\to 0.
\end{equation}
Then\begin{itemize}
\item $E$ is stable with respect to $K_Y$, $c_1(E)=K_Y+\sigma$, and $c_2(E)=1$.
\item Conversely, every $K_Y$-stable  rank 2 vector bundle with $c_1(E)= K_Y$ modulo torsion and $c_2(E)=1$ is equivalent to one of this form.   
\item  $E$ is exceptional if  $\sigma \in \Tors Y\setminus 2\Tors Y$  (in particular, we require $|\Tors Y|$ is even), and $P$ is a simple base point of $|2K_Y+\sigma|$.
\item If $E$ is exceptional, then $\sigma \in \Tors Y\setminus 2\Tors Y$
\end{itemize}

\end{thm}

 The proof breaks into three lemmas.

\begin{lem} We can define a vector bundle $E$ using the exact sequence~(\ref{K}), and such a vector bundle is stable with respect to $K_Y$. Moreover, $E$ is uniquely determined by $\sigma$ and $P$. Conversely, if $E$ is exceptional vector bundle with $c_1(E) = K_Y \mod \Tors Y$ and $c_2(E)=1$,  stable with respect to $K_Y$, then up to $E\leadsto E(\tau)$, $\tau \in \Tors Y$,  it is given by an exact sequence~(\ref{K}).
 \end{lem}
\begin{proof}   The (CB) condition in this case is satisfied for $P$ being a base point of $|K_Y+L| = |2K_Y + \sigma|$. So we can define a vector bundle $E$ using~(\ref{K}).
We need  to show that such a vector bundle $E$ is stable with respect to $K_Y$. Suppose that $E$ is not stable with respect to $K_Y$. Then since $ \mu(E)=\frac{1}{2}c_1(E).K_Y=\frac{1}{2}$, there is an effective divisor $D$ such that $D\cdot K_Y>0$ and a nonzero map $\OO_Y(D)\hookrightarrow E$. Write $D=nK_Y+\beta$, where $n= D\cdot K_Y>0$ and $\beta \in K_Y^{\perp}$. Consider the exact sequence
\begin{equation*}
0\to \OO_Y(-D)\to E(-D)\to \OO_Y(K_Y+\sigma)\otimes \OO_Y(-D)\otimes \II_P\to 0,
\end{equation*}
obtained by tensoring (\ref{K}) by $\OO_Y(-D)$.

We have $H^0(\OO_Y(-D))=0$ because $D\cdot K_Y>0$,  and $$H^0(E(-D))=\Hom (\OO_Y(D), E)\neq 0$$ by assumption, so
 \begin{equation*}
 H^0(\OO_Y(K_Y+\sigma-nK_Y-\beta)\otimes \II_P)\neq 0.
 \end{equation*}
 This is impossible for $n>1$. If $n=1$ then $H^0(\OO_Y(\sigma -\beta))=0$ unless we have $\sigma - \beta=0$. But then $H^0(\OO_Y\otimes \II_P) = H^0(\II_P)=0$. So such a divisor $D$ does not exist and $E$ is stable with respect to $K_Y$.
 
 \medskip

 Conversely, if $E$ is an exceptional vector bundle of rank 2 with $c_1(E) = K+ \sigma$, then according to the formula~(\ref{c2}), the second Chern class $c_2(E)$ is fixed and is equal to $ 1$. 
 
 We need to check that $E$ can be defined by the exact sequence~(\ref{K}) for some $\sigma$ and $P$. First, let us show that there exists a section $s\in H^0(E)$, or equivalently, that $h^0(E)>0$. By the Riemann--Roch formula $\chi(E) = 1$. Also by  Serre Duality $h^2(E) = h^0(E^{\vee}\otimes K) = h^0(E(-\sigma))$ since  have $E\simeq  E^{\vee}\otimes \det E = E^{\vee}\otimes \OO_Y(K_Y+\sigma)$.

 So $0<\chi(E)=h^0(E)-h^1(E)+h^0(E(-\sigma))\leq h^0(E)+h^0(E(-\sigma))$. Thus $h^0(E)>0$, or $h^0(E(-\sigma))>0$, so, possibly after a twist $E\leadsto E(-\sigma)$, there is  a nonzero global section of $E$.
 
 Now we need to check that a nonzero global  section of $E$ has isolated zeroes, assuming that $E$ is $K_Y$-stable. Suppose that it does not, so there is a map $\OO(D) \to E$ for some effective divisor $D$. But then $\mu(\OO(D)) = D.K\geq 1$ and $\mu(E) = \frac{1}{2}K^2= \frac{1}{2}$, which contradicts stability of  $E$. 
 
 Finally, we have $$0\to \OO_Y\to E \to L\otimes \II_Z\to 0,$$ 
 where $c_1(E)=L$, $c_2(E) = \length(Z)$. Then $L = K_Y+ \sigma$, for some $\sigma \in \Tors Y$, and $Z=P$ is a reduced point. Now by Theorem~\ref{Hlt}, $P$ is a base point of $|2K_Y + \sigma|$.

Now let us show that such $E$ is unique. It suffices to show that $\Ext^1(\OO_Y(K_Y+\sigma)\otimes \II_P, \OO_Y)\simeq \C$.  By  local-to-global spectral sequence for $\Ext$ we have an exact sequence.
\begin{equation*}
0\to H^1(\mathcal Hom )\to \Ext^1\to H^0(\mathcal Ext^1)\to H^2(\mathcal Hom)
\end{equation*}

Now $\mathcal \Hom (\OO_Y(K_Y+\sigma)\otimes \II_P, \OO_Y)= \OO_Y(-K_Y-\sigma)$, and $H^1(\OO_Y(-K_Y-\sigma))=0$ by Kodaira vanishing.
Also $\mathcal Ext^1 (\OO_Y(K_Y+\sigma)\otimes \II_P, \OO_Y)\simeq \mathcal Ext^1(\II_P, \OO_Y)\simeq k(P)$, so that $H^0(\mathcal Ext^1(\OO_Y(K_Y+\sigma)\otimes \II_P, \OO_Y))\simeq \C$. 

So we have an exact sequence
\begin{equation*}
0\to \Ext^1\xrightarrow{\alpha} H^0(\mathcal Ext^1)\to H^2(\mathcal Hom)
\end{equation*}

Since  the Cayley--Bacharach condition is satisfied,  there exists an extension which is a vector bundle. So $\alpha\neq 0$, thus $\Ext^1(\OO_Y(K_Y+\sigma)\otimes \II_P, \OO_Y)\simeq \C$. 

\end{proof}

\begin{lem}
\label{noexK}
If $c_1(E) = K_Y+\sigma$,where $\sigma \in 2 \Tors Y$,  then the vector bundle $E$ defined by the exact sequence~(\ref{K}) is not exceptional.
\end{lem}  
\begin{proof}First suppose that $c_1(E)=K_Y$. We have inclusions $\OO_Y(K_Y)\otimes\II_P\subset\OO_Y( K_Y)\subset E\otimes \OO_Y(K_Y)$ and a surjection $E\to \OO_Y(K_Y)\otimes \II_P$ coming from the exact sequence~(\ref{K}). Thus there exists a nonzero map $f:E\to E\otimes\OO_Y( K_Y)$, defined by the composition of surjection followed by inclusion. But $\Ext^2(E,E)=H^2(\mathcal Hom(E,E))= H^0(\mathcal Hom(E,E)^{\vee}\otimes \omega_Y)^*= \Hom(E,E\otimes \omega_Y)^*$ by Serre duality. Thus since $f\in \Hom(E, E\otimes \omega_Y)$, we conclude that $\Ext^2(E,E)\neq 0$, so $E$ is not exceptional.

If $\sigma = 2\tau\in \Tors Y$, then $E\otimes \OO(-\tau)$ has $c_1 = K_Y$, thus $E\otimes \OO_Y(-\tau)$ is not exceptional, so $E$ is not exceptional by Theorem~\ref{intro:thm:exc}. 
\end{proof}

\begin{lem} 
Let $Y$ be a Godeaux surface, $\sigma \in \Tors Y/2\Tors Y$, $P\in |2K_Y+\sigma|$, and $E$ the associated vector bundle defined by~(\ref{K}). Then $E$ is exceptional if  $P$ is a smooth point of the base locus of $|2K_Y+ \sigma|$.
\end{lem}
\begin{proof}

Here we have  $L =\OO_Y( K_Y + \sigma)$. By Lemma~\ref{conditions} we only need to check that conditions ~(\ref{condition1}) and ~(\ref{condition2}) are satisfied.

Let $P$ be a point in the base locus. First, we show that the condition~(\ref{condition1}) is satisfied for all Godeaux surfaces $Y$. By~\cite{Reid}, Lemma 0.3, we have  $H^0(K_Y+\sigma)\simeq \C$, therefore $|K_Y+\sigma|\neq \emptyset$.  Let $C\in |K_Y+\sigma|$, then $C$ is an irreducible curve since $K_Y\cdot C = K_Y^2=1$. Suppose that $P\in C$, then $(2K_Y+\sigma)|_C = K_C$ by the adjunction formula. 
More precisely, $\omega_C = \omega_Y^{\otimes 2}(\sigma)|_C$, where $\omega_C$ is the dualizing sheaf of $C$. 
We have an exact sequence 
$$0\to \omega_Y\to \omega_Y^{\otimes 2}(\sigma)\to \omega_C\to 0,$$
so $H^0(\omega_Y^{\otimes 2}(\sigma))\twoheadrightarrow H^0(\omega_C)$ because $H^1(\omega_Y)= H^1(\OO_Y)^*=0$.  It follows that $P$ is a base point of $\omega_C$.

By the adjunction formula $$2p_a(C)-2  = \deg \omega_C = (K_Y + C)\cdot C  = (2K_Y +\sigma)\cdot (K_Y + \sigma)=2,$$ so $p_a(C)=2$.

But by \cite{Hartshorne}, Theorem 1.6   if $C$ is a projective, irreducible Gorenstein curve with $p_a(C)>0$, then $\omega_C$ is base point free.  Thus $P\not\in C$, so 
$H^0(\OO_Y(K_Y+\sigma)\otimes \II_P) = 0$ and   $\Hom(E, E) = \C$.

\medskip

So we just need to check the condition (\ref{condition2}). By~\cite{Reid}, Lemma 0.3, we have  $\dim H^0(2K_Y+\sigma)=2$, so $|2K_Y+\sigma|\simeq \Pro^1$. Then $P$ is a simple base point of $|2K_Y+\sigma|$ if and only if there exist $C_1$ and $C_2\in |2K_Y+\sigma|$ given locally by $(x=0)$ and $(y=0)$, where $x$ and $y$ are some local coordinates at $P$, and $P=C_1\cap C_2$, so that the general $C\in |2K_Y+\sigma|$ is given locally by $( \lambda x+\mu y=0)$ for some $\lambda, \mu\in \C$. This is equivalent to $H^0(\OO_Y(2K_Y+\sigma)\otimes \II_P^2)=0$.
\end{proof}

This completes the proof of the Theorem~\ref{Kvb}.

\medskip

In the proposition below we provide an explicit construction of a Godeaux surface satisfying all assumptions of Theorem~\ref{Kvb}.

\begin{prop} There exists a Godeaux surface $Y$ with $H_1(Y, \Z)= \Z/4\Z$, such that $|2K_Y+\sigma|$  has simple base points, where $\sigma \in \Tors Y$ is a torsion element of order 4.
\end{prop}

\begin{proof}

We use the notations of~\cite{Reid} to describe such a surface.

In the weighted projective space $\Pro(1,1,1,2,2)$ with coordinates $x_1$, $x_2$, $x_3$, $y_1$, $y_3$ consider the surface $F_{4,4} = \{q_0=q_2=0\}\subset \Pro(1,1,1,2,2)$ defined by the  complete intersection of two quartics
$$q_0= x_1^4+ x_2^4+x_3^4 + y_1y_3,$$
$$q_2 = x_1^3x_3+ x_1x_3^3 +x_1x_2y_3+  y_1^2+ y_3^2.$$
 Consider the action of $\Z/4\Z$ on $F_{4,4}$ induced from the action of $\Z/4\Z$ on the weighted projective space $\Pro(1,1,1,2,2)$  defined by $x_i \mapsto \zeta^i x_i$, $y_i\mapsto \zeta^i y_i$, where $\zeta$ is a primitive fourth root of unity.

 Then the quotient of  $F_{4,4}$ under the $\Z/4\Z$ action is a smooth Godeaux surface $Y$ with $H_1(Y)= \Z/4\Z$~\cite{Reid}. 
For this surface and $\sigma = 1\in \Z/4\Z$,  the linear system $|2K_Y+\sigma|$ has simple base points, so the condition~(\ref{condition2}) is satisfied for this particular surface. 
  Note that $\{q_0=q_2=0\}$ define the smooth \'etale cover $F_{4,4}$ of a Godeaux surface $Y$ in $\Pro(1,1,1,2,2)$.
The base locus of $|2K+\sigma|$ is the union of the sets $\{x_3=y_1=0\}\cap F_{4,4}$ and $\{x_2=y_1=0\}\cap F_{4,4}$. It is not hard to show that it consists of 16 distinct points in $\Pro(1,1,1,2,2)$, namely $(1, \zeta, 0,0,0)$, $(1, \zeta, 0,0,-\zeta)$, $(1,0, \zeta, 0,\pm\sqrt{-\zeta-\zeta^3})$, where $\zeta^4 = -1$.

Thus its quotient under the action of $\Z/4\Z$  consists of four distinct points, so $H^0( (2K+\sigma)\otimes \II_P^2)=0$. 
\end{proof}

Thus a vector bundle $E$ of Theorem~\ref{Kvb} is exceptional for a general Godeaux surface $Y$ with $\Tors Y = \Z/4\Z$. 
In the case $\Tors Y = \Z/2\Z$ it becomes significantly harder to write out equations for a surface. 
Stephen Coughlan showed by explicit calculation in Macaulay2 that condition (\ref{condition2})  is satisfied for a least one Godeaux surface, so that it holds in general, i.e, on a Zariski open subset of the irreducible component of moduli space containing Coughlan's surface.

Note that using the Theorem \ref{Kvb} we obtain four non-isomorphic vector bundles with the same rank, $c_1$, and $c_2$, based on the choice of the point $P$. Since $H^0(\OO_Y(K_Y+\sigma)\otimes\II_P)=0$, using the exact sequence (\ref{K}) we conclude  that $H^0(E) \simeq H^0(\OO_Y) \simeq \C$, thus there is exactly one section $s\in H^0(E)$ and $P$ is uniquely determined as  the zero locus of this section. 

\begin{cor} Let $Y$ be a general Godeaux surface with $H_1(Y, \Z)=\Z/4\Z$ and $\sigma \in \Tors Y$ is an element of order 4.
There exist at least four isomorphism classes of  exceptional vector bundles of rank 2 on $Y$ with  $c_1= K_Y+\sigma$ which are stable with respect to $K_Y$.
\end{cor}

In the  case $c_1(E)\not\sim K_Y \mod\Tors Y$, we provide  method of constructing non--stable exceptional vector bundles of rank 2 with $c_1(E)\not \sim K_Y \mod \Tors Y$.

\begin{thm}\label{nonstablebundle} Let $Z$ be a special Godeaux surface such that  $Z$ contains a $(-3)$-curve $C$. Let $Y$ be a small deformation of $Z$ such that $C$ does not deform to $Y$.  
 Define the line bundle  $L$ as the inverse  image of the line bundle   $\OO_Z(-C)$ under the isomorphism $H^2(Y, \Z) \to H^2(Z, \Z)$ given by specialization. Then there exists a unique extension 
\begin{equation}\label{def:f}
0\to \OO_Y\to F\to L\to 0,
\end{equation} 
where $F$ is an exceptional vector bundle of rank 2 on $Y$, but it is not $K_Y$--stable. 
\end{thm}
\begin{proof} 
By the definition of $L$ we have $L^2=-3$ and $L\cdot K_Y = -1$. Also we can compute $\chi (L) = 1 + (1/2)L \cdot (L-K_Y)=0$, and $\chi(L^{\vee}) = 1+ 1/2 (-L)\cdot (-L-K_Y) = -1$.

We claim that the line bundle $L$ satisfies $H^0(L)=H^1(L)=H^2(L)=0$, and $H^0(L^{\vee})=H^2(L^{\vee})=0$, $H^1(L^{\vee})=\C$.

Indeed, $H^0(L)=0$, since $L\cdot K_Y=-1$, and $K_Y$ is ample. On $Z$ we have $(K_Z+C)\cdot C = 1-3=-2<0$. Thus if $D\in |K_Z+C|$, then we have $D=D'+C$, where $D'$ is effective. So $D'\in |K_Z|$, which is impossible since $|K_Z|=\emptyset$.  Also by semicontinuity of cohomology~\cite{Har77}, Ch III, Theorem 12.8, the fact that $H^0(\OO_Z(K_Z+C))=0$ on $Z$ implies that $H^0(\OO_Y(K_Y)-L)=0$ on $Y$. 
We conclude that $H^1(L)=0$ using the fact that $\chi(L)=0$. 

To show that $H^0(L^{\vee})=0$, note that if $D\in H^0(L^{\vee})$, then $D^2=-3$, $D\cdot K_Y=1$, so $D$ is an irreducible curve on $Y$. By the adjunction formula $$2p_a(D)-2=D\cdot (D+K_Y)=-2,$$ so that $p_a(D)=0$.
Then $D\simeq \Pro^1$, and by our assumption $C$ does not deform to a $(-3)$-curve on $Y$.  This is a contradiction, thus $H^0(L^{\vee})=0$. Finally $H^2(L^{\vee})= H^0(\OO_Y(K_Y+L))^*$ by Serre Duality.
We have $K_Y\cdot (K_Y+L)=0$, but $K_Y+L\not\sim 0$, since for example $L\cdot (L+K_Y)=-4\neq 0$, so $H^0(\OO_Y(K_Y+L))^*=0$ since $K_Y$ is ample. Finally since $\chi(L^{\vee})=-1$, we conclude that $H^1(L^{\vee})\simeq \C$.

Now we notice that the Cayley--Bacharach property (CB) is satisfied here automatically, since $Z=\emptyset$. We compute $\Ext^1(L, \OO_Y) = H^1(L^{\vee})\simeq \C$, so the extension~(\ref{def:f}) exists, and it is unique. 
It is easy to see that $F$ is not stable here, since $c_1(F)=L$, and $L\cdot K_Y=-1$.

So it only remains to show that $F$ is exceptional.
We can apply $\Hom$ from the exact sequence~(\ref{def:f}) to itself to obtain the commutative diagram shown in Figure~\ref{figure2}.

\begin{figure}
\begin{equation*}
\begin{CD}
@.0@.0@.0\\
@.@AAA@AAA@AAA\\
@. \stackrel{b_3}{\longrightarrow} \Ext^2(\OO_Y, \OO_Y)@>>>\Ext^2(\OO_Y, F)@>>>\Ext^2(\OO_Y, L) \rightarrow0\\
@.@AAA@AAA@AAA\\
@. \stackrel{b_2}{\longrightarrow} \Ext^2(F, \OO_Y)@>>>\Ext^2(F, F)@>>>\Ext^2(F, L) \rightarrow0\\
@.@AAA@AAA@AAA\\
@. \stackrel{b_1}{\longrightarrow} \Ext^2(L, \OO_Y)@>>>\Ext^2(L, F)@>>>\Ext^2(L, L) \rightarrow 0\\
@.@AAA@AAA@AAA\\
@. \stackrel{a_3}{\longrightarrow} \Ext^1(\OO_Y, \OO_Y)@>>>\Ext^1(\OO_Y, F)@>>>\Ext^1(\OO_Y, L)\stackrel{b_3}{\longrightarrow} \\
@.@AAA@AAA@AAA\\
@. \stackrel{a_2}{\longrightarrow} \Ext^1(F, \OO_Y)@>>>\Ext^1(F, F)@>>>\Ext^1(F, L)\stackrel{b_2}{\longrightarrow} \\
@.@AAA@AAA@AAA\\
@. \stackrel{a_1}{\longrightarrow} \Ext^1(L, \OO_Y)@>>>\Ext^1(L, F)@>>>\Ext^1(L, L) \stackrel{b_1}{\longrightarrow} \\
@.@A\delta AA@AAA@AAA\\
@.0\rightarrow\Hom(\OO_Y, \OO_Y)@>>>\Hom(\OO_Y, F)@>>>\Hom(\OO_Y, L) \stackrel{a_3}{\longrightarrow} \\
@.@AAA@AAA@AAA\\
@.0\rightarrow\Hom(F, \OO_Y)@>>>\Hom(F, F)@>>>\Hom(F, L) \stackrel{a_2}{\longrightarrow} \\
@.@AAA@AAA@AAA\\
@.0\rightarrow\Hom(L, \OO_Y)@>>>\Hom(L, F)@>>>\Hom(L, L) \stackrel{a_1}{\longrightarrow} \\
@.@AAA@AAA@AAA\\
@.0@.0@.0
\end{CD}
\end{equation*}
\caption{Commutative diagram obtained from applying $\Hom$ from the exact sequence~(\ref{def:f}) to itself}
\label{figure2}
\end{figure}

Now $\Hom(\OO_Y, L)= \Ext^1(\OO_Y, L)=\Ext^2(\OO_Y,L)=0$, since $H^i(L)=0$ for all $i$. Also $\Hom(\OO_Y, \OO_Y)=\C$, and $\Ext^i(\OO_Y, \OO_Y)=0$, for $i=1,2$. Since $\Ext^i(F,G) = H^i(F^{\vee}\otimes G)$, we conclude that $\Hom(L, L)\simeq \C$, and $\Ext^i(L, L)=0$ for $i=1,2$. Also $\Hom(L,\OO_Y)=H^0(L^{\vee})=0$, $\Ext^1(L, \OO_Y)\simeq \C$, and $\Ext^2(L, \OO_Y)=0$.

Now since $\Ext^2(F, \OO_Y) = \Ext^2(F, L)=0$, we conclude that $\Ext^2(F, F)=0$.
Also $\Ext^1(F, L)=0$, and the map $\delta\neq 0$, since $\delta(id)=e\in \Ext^1(L, \OO_Y)\simeq \C$  is the extension class~\cite{Har77}, Chapter III, Exercise 6.1, and, moreover, since $\dim \Ext^1(L, \OO_Y)=1$, $\delta$ is an isomorphism. It implies that 
\begin{equation*}\Hom(F, \OO_Y)=\Ext^1(F, \OO_Y)=0.
\end{equation*}
 So $\Ext^1(F, F)=0$, and $\Hom(F, F)\simeq \Hom(F, L)\simeq \Hom(L, L)\simeq \C$. Thus $F$ is exceptional. 
\end{proof}

\begin{remark}\label{rk:flip}
Note that the vector bundle $F$ from  Theorem~\ref{nonstablebundle} can be obtained by applying the construction of Hacking (Theorem~\ref{thm:hacking}) to  birational modification of the family $\mathcal Z/\Delta$, with special fiber $Z$ and general fiber $Y$.

We can blow $Z$ up at a point of $C$, to obtain a surface $\X$ containing a chain of two $\Pro^1$'s with self intersections $(-4)$, and $(-1)$, intersecting at one point. Now we can contract the $(-4)$ curve on $\X$ to obtain a surface $X$ with a unique $\frac{1}{4}(1,1)$ singularity. Let $\Gamma$ be the image of $(-1)$--curve on $X$.  Then the birational map $X\dashrightarrow Z$, $X\backslash \Gamma \simeq Z\backslash C$ extends to a birational map  $\mathcal X\dashrightarrow \mathcal Z$ over $\Delta$, with $\mathcal X\backslash\Gamma \simeq \mathcal Z\backslash C$, where $\mathcal X$ is a 3-fold with terminal singularities.  
   Thus according to ~\cite{Ha11}, we can construct a reflexive sheaf $\mathcal E$ on $\mathcal X$ such that its restriction $F$ on a nearby smooth fiber $Y$ is exceptional, and $sp(c_1(F))= 2\Gamma \in H_2(X, \Z)$, where $sp: H_2(Y,\Z)\to H_2(X, \Z)$ is a specialization map. Then $c_1(F)=L \in H_2(Z,\Z)$, and one can show that $F$ is isomorphic to the bundle constructed in Theorem~\ref{nonstablebundle}.
\end{remark}

We now show that a degeneration $Y\leadsto Z$ as in Theorem~\ref{nonstablebundle} exists for a Godeaux surface $Y$ with $\Tors Y= \Z/5\Z$.
\begin{thm}
There exists a Godeaux surface $Y$ with $H_1(Y,\Z)\simeq \Z/5\Z$, containing no $(-3)$-curves. 
\end{thm}
\begin{proof}
Suppose that a Godeaux surface $Y$ with $H_1(Y, \Z) = \Z/5\Z$ contains a $(-3)$ curve $C$. Then by adjunction formula $K_Y\cdot C$=1. On the \'etale cover $\phi:Z\to Y$ of $Y$ corresponding to $\Tors Y=\Z/5\Z$ the preimage $\phi^{-1}(C)$ therefore has to be a union of five disjoint copies of $(-3)$ curves, i.e.  five lines on a quintic surface. 

Let us show that these lines do not appear on a general $\Z/5\Z$-invariant quintic surface. We can list all the lines explicitly for the Fermat quintic and then check by a first order deformation theory calculation that these lines do not deform to a nearby general surface. 

Let $S=(F=0) \subset \Pro^3$ be a Fermat quintic, and let $L = (X=Y=0)$ be a line on $S$.  
Then we can write $F$ as $F = AX + BY$ for some quartic  forms $A$ and $B$.

Consider a family $\mathcal S\to T$, where $T = k[t]/(t^2)$ such that $\mathcal S_0 = S$ and $\mathcal S_t = (F+tG=0)\subset \Pro^3$, where $G$ is $\Z/5\Z$ invariant.
Then the line  $L\subset S$ deforms to $\mathcal L\subset \mathcal S$
over $k[t]/(t^2)$ if and only if   we can write $(F+tG)$ as $\big((A+tC)(X+tZ)+(B+tD)(Y+tW)\big)$ for some $C,D,Z,W$ of corresponding degrees 4,4,1,1. If this is possible, then we can write $G=AZ+CX+BW+DY$, so $G\in (X,Y,A,B)$.
A Macaulay calculation shows that  a general  $\Z/5\Z$-invariant $G$  cannot be written in this form for any of  the lines on  the Fermat quintic. Thus there are no $(-3)$ curves on a general Godeaux surface $Y$ with $H_1(Y) = \Z/5\Z$.\\
\end{proof}


\section{Correspondence}

This section summarizes our results relating to the correspondence~(\ref{ch}) between degenerations and vector bundles.

\begin{thm}\label{c:thm:2d}  Let $Y\leadsto X$ be a  $\Q$-Gorenstein degeneration, where $Y$ is a Godeaux surface, and $X$ has a unique singularity of type $\frac{1}{4}(1,1)$ and $K_X$ is ample. Assume $H_1(Y, \Z)\simeq H_1(X, \Z)$. Let $\sigma \in H_1(Y)$. The construction of Hacking  produces an exceptional vector bundle $E$ of rank $2$ on $Y$ with $c_1(E) = K_Y+\sigma$ modulo the equivalence relation if and only if $K_X+\sigma\in H_2(Y, \Z)$ is a 2-divisible divisor on $X$. 
\end{thm}

\begin{proof}
Suppose $K_X+\sigma$ is 2-divisible where $\sigma \in H_1(X)= \Tors H^2(X)= \Tors \Pic X$ is torsion. Write $K_X+\sigma =2D$. 

The local class group of the singularity $P \in X \simeq \frac{1}{4}(1,1)$ is isomorphic to $\Z/4\Z$. Here $w \in \Z/4\Z$ corresponds to the class of a divisor given by an equation $(f(u,v)=0) \subset \C^2_{u,v}/\frac{1}{4}(1,1)$ where $f$ has weight $w$ with respect to the group action, i.e., under the generator $(u,v) \mapsto (\zeta u, \zeta v)$ we have $f \mapsto \zeta^w f$. 
Note that the local class of $K_X+\sigma$ corresponds to $2 \in \Z/4\Z$.
(Indeed, a local section of $\omega_X=\OO_X(K_X)$ is given by $\Omega=f(u,v) du \wedge dv$, where $f$ has weight 2 (so that $\Omega$ is invariant with respect to the group action). Then its zero locus $(\Omega=0)=(f(u,v)=0)$ is a divisor in the class $K_X$.)  Also $\sigma$ is a Cartier divisor (i.e., corresponds to a line bundle) so $\sigma = 0$ in the local class group. 
Thus $D$ corresponds to $\pm 1 \in \Z/4\Z$, equivalently, $\pm D$ is locally linearly equivalent to $(v=0)$. Now by~\cite{Ha11}, Proposition 4.2, we obtain an exceptional bundle $F$ on $Y$ with $c_1(F) = \pm 2D = \pm( K_Y+\sigma)$. Replacing $F$ by $F^*$ if necessary we obtain $c_1(F)= K_Y+\sigma$.

Conversely, suppose $F$ is an exceptional bundle with $c_1(F)=K_Y+\sigma$ associated to a degeneration. Then we have $c_1(F)=2D \in H_2(X,\Z)=\Cl(X)$ for some divisor $D$ on $X$ by~\cite{Ha11}, Theorem 1.1
\end{proof}

\begin{prop} \label{prop:torsX}Let $Y\leadsto X$ be a degeneration of a smooth Godeaux surface $Y$ to a surface $X$ with a unique singularity $(P\in X)$ of Wahl type $\frac{1}{4}(1,1)$. Assume that $H_1(Y)\simeq H_1(X)$. Then 
$$\Tors H_2(X)\simeq \Tors H^2(X)\simeq \Tors H^2(Y).$$
\end{prop}
\begin{proof}

Write $X = X^0\cup_L C$, where $L$ is the link of singularity, and $C$ is the cone of singularity, and $X^0 = X\backslash C$.  Then there is a  Mayer--Vietoris sequence
$$H_2(X)\to H_1(L)\to H_1(X^0)\oplus H_1(C)\to H_1(X)\to 0.$$

Since $H_1(Y)\simeq H_1(X)$, by~\cite[p.~134]{Ha13},  we obtain $H_2(X)\twoheadrightarrow H_1(L)$, thus $H_1(X^0)\simeq H_1(X)$ since $C$ is contractible.

By the Universal coefficient theorem $\Tors H^2(X^0)\simeq \Tors H_1(X^0) = H_1(X)$ since $H_1(X^0) = H_1(X)$ is torsion, and $\Tors H^2(Y)\simeq \Tors H_1(Y) = H_1(Y)$.

Now by Poincare Duality $\Tors H_2(X)=\Tors H^2(X^0)$. 
So $\Tors H_2(X)\simeq H_1(X)$. Thus $\Tors H_2(X)\simeq   \Tors H^2(Y)$.
\end{proof}

\begin{prop} \label{c-2-div}In the classification of the minimal resolutions $\X$ of the $\Q$-Gorenstein degenerations $Y\leadsto X$ in Theorem~\ref{cases} the canonical class $K_X+\sigma$ is 2-divisible in $H_2(X)$ for some $\sigma\in \Tors X$ precisely in the cases (a) and (c). 
\end{prop}
\begin{proof}  We have $\Tors H_2(X, \Z)\simeq \Tors H^2(Y, \Z)$.   So $K_X+\sigma\in H_2(X)$ is 2-divisible for some $\sigma \in \Tors H_2(X)=\Tors H^2(Y)$ if and only if $K_X\in H_2(X)/\Tors H_2(X)$ is 2-divisible.

Let $E\simeq \Pro^1$ be the exceptional locus of  the minimal resolution $\pi:\X\to X$ of $X$. Then $\X = X^0 \cup N$, where $N =\pi^{-1} C $ is homotopy equivalent to $E\simeq\Pro^1$.
We have an exact sequence
\begin{equation}\label{H_2}
0\to \Z\to H_2(\X, \Z)\to H_2(X, \Z)\to 0,
\end{equation}
where the first map is given by $1\mapsto [E]$. Since $\Tors H_2(\X)=\Tors H_2(X)$, the exact sequence~(\ref{H_2}) is split, i.e. $H_2(\X, \Z)\simeq \Z\cdot [E] \oplus H_2(X, \Z)$, defined by
 $\alpha\mapsto (\theta(\alpha)\cdot [E]\oplus \pi_* \alpha)$ for some $\theta: H_2(\X)\to\Z$.

Now $\pi_*K_{\X}= K_X$. So $K_X$ is 2-divisible in $H_2(X, \Z)/\Tors H_2 (X, \Z)$ if and only if $K_{\X}$ or $K_{\X}+E$ is 2-divisible in $H_2(\X, \Z)$.

Note that $K_{\X}+E$ is not 2-divisible in $H_2(\X, \Z)$. Indeed, if $K_{\X}+E= 2D$ for some $D\in H_2(\X, \Z)$, then $K_{\X}\cdot D = 1$ and $D^2=0$, which is a contradiction since we always have $K_{\X}\cdot D \equiv D^2 \mod 2$.

So it only remains to check when $K_{\X}$ is divisible in $H_2(\X, \Z)/\Tors H_2(\X, \Z)$. 

Let $K_{\X} = \lambda A$, where $A$ is a general fiber of the elliptic fibration $\X\to \Pro^1$. Let $m =\lcm (m_1, m_2)$. 
By the Kodaira Canonical Bundle formula~(\ref{cbf}) we have $$\lambda=1-\frac{1}{m_1}-\frac{1}{m_2}.$$ 

Clearly, if $\mu \in \Q$, $\mu m \in \Z$, then $\mu A\in H_2(\X, \Z)$. 

Conversely, by~\cite{FM}, Chapter II, Proposition 2.7, there exists some $D\in H_2(\X, \Z)$ such that $D\cdot A = m$. Thus if $\mu A\in H_2(\X, \Z)$, then $D\cdot \mu A = \mu m$ is an integer. 
So $K_{\X}$ is 2-divisible in $H_2(\X, \Z)/\Tors H_2(\X, \Z)$ if and only if $\lambda m \in 2\Z$.

 We have:
\begin{enumerate}
\item[(a)] $m_1=4$, $m_2=4$,  $\lambda m= 2$, $K_{\X} = \frac{1}{2} A= 2F_4$;
\item [(b)]$m_1=3$, $m_2=3$,  $\lambda m= 1$,  $K_{\X} = \frac{1}{3} A = F_3$;
\item [(c)]$m_1=2$, $m_2=6$,   $\lambda m= 2$,  $K_{\X} = \frac{1}{3} A = 2F_6$;
\item [(d)]$m_1=2$, $m_2=4$,  $\lambda m= 1$,   $K_{\X}=\frac{1}{4} A = F_4$;
\item [(e)]$m_1=2$, $m_2=3$,  $\lambda m= 1$,  $K_{\X}= \frac{1}{6}A=F_2-F_3$.
\end{enumerate}
Thus $K_X$ is 2--divisible in $H_2(X)/\Tors X$ in cases (a) and (c). 
\end{proof}

\begin{proof}[Proof of the Theorem \ref{intro-c}.]
In the case $H_1(Y, \Z) = \Z/4\Z$, a degeneration $Y\leadsto X$, where $X$ has a unique Wahl singularity of type $\frac{1}{4}(1,1)$ is explicitly constructed in Proposition~\ref{p4}. By Proposition~\ref{c-2-div}, $K_X$ is 2--divisible in $H_2(X)/\Tors X$. According to the Theorem~\ref{c:thm:2d},  we can produce an exceptional vector bundle $E$ of rank $2$ on $Y$ with $c_1(E) = K_Y+\sigma$ using the construction described in  Theorem~\ref{thm:hacking}. Every such vector bundle is equivalent to the vector bundle defined in Theorem~\ref{Kvb}. 

\end{proof}

 We were not able to construct  exceptional vector bundles $E$ of rank 2 on $Y$ such that $c_1(E)\not\sim K_Y$, and $E$ is stable with respect to $K_Y$ directly. The construction of Hacking guarantees  the existence of such vector bundles in the cases $H_1(Y, \Z)= \Z/3\Z$, $H_1(Y, \Z)=\Z/2\Z$ or $H_1(Y, \Z)=0$.

The main open problem is to determine whether there are exceptional vector bundles of rank 2  on $Y$ such that $c_1(E)\not\sim K_Y$, and such that $E$ is stable with respect to $K_Y$ in the cases $H_1(Y,\Z)= \Z/4\Z$ and $H_1(Y, \Z) = \Z/5\Z$. 
If such vector bundles exist, then they cannot possibly come from  degenerations, because no such boundary components exist in the classification of Theorem~\ref{cases}, and Proposition~\ref{c-2-div}. Interestingly, in the case $H_1(Y)=\Z/5\Z$ there do exist exceptional vector bundles of rank 2 which are not $K_Y$-stable, and moreover these can be obtained from degenerations to a surface $X$ with $\frac{1}{4}(1,1)$ singularity for which $K_X$ is not nef, as shown in  Theorem \ref{nonstablebundle}.

\end{document}